\documentclass[a4paper,12pt,reqno]{amsart}
\usepackage{amsmath, amsfonts, amssymb}
\usepackage[hang,flushmargin]{footmisc} 
\usepackage{amsthm}
\usepackage{enumitem}
\usepackage{comment}
\usepackage{hyperref}
\usepackage[dvistyle, colorinlistoftodos]{todonotes}

\newtheorem{theorem}{Theorem}[section]
\newtheorem{lemma}[theorem]{Lemma}

\newtheorem{proposition}[theorem]{Proposition}

\newtheorem{example}[theorem]{Example}
\theoremstyle{definition}
\newtheorem{defn}[theorem]{Definition}

\def\Z{\mathbb{Z}}
\def\R{\mathbb{R}}

\def\P{\mathbb{P}}

\newcommand{\Zmod}[1]{\Z_{#1}} 

\let\originalleft\left
\let\originalright\right
\renewcommand{\left}{\mathopen{}\mathclose\bgroup\originalleft}
\renewcommand{\right}{\aftergroup\egroup\originalright}

\begin{document}

\title{On the dimension of additive sets}

\author{P. Candela}
\address{D\'epartement de math\'ematiques et applications\newline
	\indent \'Ecole normale sup\'erieure, 
	Paris,
	France}
\email{pablo.candela@ens.fr}
\author{H. A. Helfgott}
\address{D\'epartement de math\'ematiques et applications\newline
	\indent \'Ecole normale sup\'erieure, 
	Paris,
	France}
	\email{helfgott@dma.ens.fr}
\thanks{Research supported by project ANR-12-BS01-0011 CAESAR and by a postdoctoral grant of the \'Ecole normale sup\'erieure, Paris.}
\subjclass[2010]{Primary 11B30; Secondary  05D40}
\keywords{Additive dimension, dissociated sets}
\maketitle

\begin{abstract}
We study the relations between several notions of dimension for an additive set, some of which are well-known and some of which are more recent, appearing for instance in work of Schoen and Shkredov. We obtain bounds for the ratios between these dimensions by improving an inequality of Lev and Yuster, and we show that these bounds are asymptotically sharp, using in particular the existence of large dissociated subsets of $\{0,1\}^n\subset \Z^n$.
\end{abstract}

\section{Introduction}
Let $A$ be an additive set, that is, a finite subset of an abelian group. A \emph{subset sum} of $A$ is a sum of the form $\sum_{a\in A'} a$ for some set $A'\subset A$. By a \emph{$[-1,1]$-combination of} $A$, we mean a sum $\sum_{a\in A} \varepsilon_a\, a$ with coefficients $\varepsilon_a$ lying in $ [-1,1]=\{-1,0,1\}$. 

\begin{defn}
A subset $D$ of an abelian group is said to be \emph{dissociated} if the subset sums of $D$ are pairwise distinct; equivalently, the only $[-1,1]$-combination of $D$ that equals 0 is the one with all coefficients equal to 0. We say that $D$ is a \emph{maximal} dissociated subset of $A$ if there is no dissociated set $D'\subset A$ such that $D'\supsetneq D$.
\end{defn}

Dissociativity plays an important role in additive combinatorics and harmonic analysis; see \cite{TomNotes} and \cite[\S 4.5]{T-V}. In particular, it provides an analogue, in the setting of general abelian groups, of the concept of linear independence from linear algebra, and it is often used to define a notion of dimension for an additive set. For a recent instance, in the work of Schoen  and Shkredov \cite{SS} the terminology `additive dimension of $A$' is used for the maximal cardinality of a dissociated subset of $A$. We shall call this quantity the dissociativity dimension.

\begin{defn}
Let $A$ be an additive set. We define the \emph{dissociativity dimension} of $A$ to be the number 
$d_d(A):=\max \{|D|: D\subset A,\; D\textrm{ is dissociated}\}$. We say that $D$ is a \emph{maximum} dissociated subset of $A$ if  $|D|=d_d(A)$. We also define the \emph{lower dissociativity dimension} of $A$ to be the number 
$d_d^-(A):=\min \{|D|: D\subset A \textrm{ is maximal dissociated}\}$.
\end{defn}

The variant $d_d^-(A)$ is considered less often than $d_d(A)$ in the literature; it appears for instance in \cite[Section 8]{SS}, where it is denoted by $\tilde d(A)$.

In linear algebra, the concepts of linear independence and dimension are linked to that of a linear-span. The well-known  basic result is that in a vector space the maximum cardinality of a linearly independent set, if finite, is equal to the minimum cardinality of a spanning set, the resulting number being by definition the dimension of the space. In the more general context of additive sets, there is an analogue of the linear span, related to dissociativity. We define it and give a corresponding notion of dimension, as follows.

\begin{defn}
Given a subset $S$ of an abelian group $G$, the 1\emph{-span} of $S$,  denoted $\langle S \rangle$, is the set of all $[-1,1]$-combinations of $S$. Given a subset $A\subset G$, we shall call a set $S\subset G$ satisfying $\langle S\rangle \supset A$ a \emph{1-spanning set} for $A$. We define the  \emph{1-span dimension} of an additive set $A$ to be the number 
$d_s(A):=\min \{|S|: S\subset A,\; \langle S\rangle \supset A\}$.
\end{defn}

This quantity has also been considered in \cite[Section 8]{SS}, where it is denoted $d(A)$. A variant of this notion, which can be called the \emph{lower 1-span dimension} of $A$, is the number $d_s^-(A):=\min \{|S|: S\subset G,\; \langle S\rangle \supset A\}$; here $G$ is the ambient abelian group containing $A$ and the sets $S$ are allowed to have elements in $G\setminus A$. This variant also appears in \cite{SS}, where it is denoted $d_*(A)$. It had already appeared in previous works, notably as the number denoted $\ell(A)$ in \cite{Sch}.

Given the basic result from linear algebra recalled above, it is natural to compare the numbers $d_d(A),d_d^-(A)$ with $d_s(A),d_s^-(A)$. It follows promptly from the definitions that if $D$ is a maximal dissociated subset of $A$ then $\langle D\rangle \supset A$. We then  deduce that
\[
d_s^-(A)\leq d_s(A)\leq d_d^-(A)\leq d_d(A).
\]
In contrast to the linear-algebra setting, each of these inequalities can be a strict one. In this paper we study the extent to which these quantities can differ from each other.

Our first result is the following lower bound on the ratio $d_s^-(A) / d_d(A)$.

\begin{theorem}\label{thm:main}
Let $A$ be an additive set. Then we have
\begin{equation}\label{eq:mainineq}
\frac{d_s^-(A)}{d_d(A)}\;\geq \; \frac{1}{\log_4 d_d(A)}\; \big(1+o(1)_{d_d(A)\to \infty}\big).
\end{equation}
\end{theorem}

We deduce this from an inequality relating the size of an arbitrary 1-spanning set for $A$ to the size of an arbitrary dissociated subset of $A$; see Proposition \ref{lem:dslb}. This inequality can be viewed as a refinement of an inequality of Lev and Yuster, namely inequality $(*)$ in \cite[Proof of Theorem 2]{LY}.

It is then natural to wonder whether there exist  additive  sets for which the ratio $d_s^-/d_d$ reaches the lower bound given by \eqref{eq:mainineq}, and more precisely whether each of the ratios of consecutive dimensions, i.e. $d_s^-/d_s,d_s/d_d^-,d_d^-/d_d$ can reach this lower bound.

For each positive integer $n$, let $Q_n$ denote the discrete cube $\{0,1\}^n$ viewed as an additive set in $\Z^n$. It follows from known results that $d_d(Q_n)=n\log_4 n\;(1+o(1))$ as $n\to \infty$. This was established independently by Lindstr\"om \cite{Lind} and by Cantor and Mills  \cite{C&M}; the result is related to the \emph{coin weighing problem}, and similar results have been treated in other works (for a recent treatment, providing several references, see \cite{Bs}).

Let $D_n$ be a dissociated subset of $Q_n$ of cardinality $|D_n|=d_d(Q_n)$. Since the standard basis is itself a maximal dissociated subset of $Q_n$ of minimum size $n$, the set $Q_n$ shows that the ratio $d_d^-(A)/d_d(A)$ can be as small as $1/\log_4 d_d(A)$ asymptotically as $d_d(A)\to\infty$. Hence the lower bound in \eqref{eq:mainineq} is asymptotically sharp.  Moreover, this set $D_n$ itself is an example showing that $d_s^-(A)/d_s(A)$ can also be as small as $1/\log_4 d_d(A)$, since for $D_n$ we have $d_s^-(D_n)=n$ yet $d_s(D_n)=|D_n|=d_d(D_n)$ (as $D_n$ is dissociated). Our second result completes the picture by showing that the remaining ratio $d_s(A)/d_d^-(A)$ can also be this small.

\begin{theorem}\label{thm:midratio}
For each positive integer $n$ there exists a set $A_n\subset \{0,1,2\}^n$ satisfying $d_d(A_n)= n\log_4 n\, (1+o(1)_{n\to \infty})$ and such that
\begin{equation}\label{eq:midratio}
\frac{d_s(A_n)}{d_d^-(A_n)}\leq \frac{1}{\log_4 d_d(A_n)}\big(1+o(1)_{n\to\infty}\big).
\end{equation}
\end{theorem}

Theorems \ref{thm:main} and \ref{thm:midratio} are  proved in Section \ref{section:main}.

In Section \ref{section:[N]} we consider sets of integers to examine whether, for at least some nice family of subsets of $\Z$, we have that for every set $A$ in the family the dissociativity dimensions $d_d(A)$, $d_d^-(A)$ lie closer to the spanning dimensions $d_s(A),d_s^-(A)$ than is guaranteed by \eqref{eq:mainineq}. The family of  intervals $[N]=\{1,2,\ldots,N\}$ is a natural one to consider; let us recall for instance (see \cite[p. 59]{Erdos-Graham}) that it is one of the oldest problems of Erd\H os to prove that $d_d([N])=\log_2 N + O(1)$. We do not pursue that problem here, but we prove the following.

\begin{theorem}\label{thm:[N]}
For any positive integer $N$ we have
\[
d_s([N])=d_d^-([N])= \lfloor \log_3 N\rfloor + \big\lceil \log_3 2N -\lfloor \log_3 N\rfloor \big\rceil.
\]
\end{theorem} 
In the final section we briefly describe a relation between the dimension $d_s$ and a result of Schoen on maximal densities of subsets of $\Zmod{p}$ avoiding solutions to a linear equation with integer coefficients.

\section{On general additive sets: Theorems \ref{thm:main} and \ref{thm:midratio} }\label{section:main}

Given an additive set $A$, a 1-spanning set $S\subset A$ has size bounded below trivially by $\log_3(|A|)$, since  $|\{-1,0,1\}^{S}|\geq |A|$. The argument leading to inequality $(*)$ in \cite[Proof of Theorem 2]{LY} is easily adapted to yield the following lower bound for $|S|$: we have $|S|\geq |D| / \log_2(2|D|+1)$ for every dissociated set $D\subset A$. This lower bound can be strengthened as follows.

\begin{proposition}\label{lem:dslb}
Let $A$ be a finite subset of an abelian group $G$, let $D\subset A$ be dissociated, and let $S\subset G$ be a 1-spanning set for $A$. Then 
\begin{equation}\label{eq:dslb}
\frac{|D|}{\log_4|D|} \leq  |S| \left(1+ \frac{4+\log_2 \log 4|S|}{\log_2 |D|}\right).
\end{equation}
\end{proposition}
Theorem \ref{thm:main} follows from this, since $d_s^-(A)\leq d_d(A)$.
\begin{proof}
Let $m=|S|, n=|D|$, and let us fix a labelling of the elements of $S$ and $D$, thus
$S=\{s_1,s_2,\ldots, s_m\}$ and $D=\{d_1,d_2,\ldots,d_n\}$. Since $\langle S\rangle \supset A\supset D$, for each $j\in [n]$ we can fix a choice of a vector $(c_{i,j})_{i\in [m]}\in \{-1,0,1\}^m$ such that $d_j= \sum_{i\in [m]} c_{i,j} s_i$. Let $C$ be the $m\times n$ matrix with $(i,j)$ entry $c_{i,j}$.

The subset sums of $D$ are the combinations $\sum_{j=1}^{n} \lambda_j d_j$ with $\lambda=(\lambda_j)\in \{0,1\}^n$. We have
\begin{equation}\label{eq:main1}
\forall\, \lambda\in \{0,1\}^n,\qquad \sum_{j\in [n]} \lambda_jd_j = \sum_{i\in [m]} \Big(\sum_{j\in [n]} c_{i,j}\, \lambda_j \Big) s_i
=  \sum_{i\in [m]}  (C\lambda)_i\,s_i .
\end{equation}

We shall prove that, for some intervals of integers $\Lambda_1,\Lambda_2,\ldots, \Lambda_m$, each of width $O\Big(\sqrt{|D|\log |S|}\Big)$, for a large proportion of the elements $\lambda\in \{0,1\}^n$ we have $(C\lambda)_i \in \Lambda_i$ for every $i\in [m]$. To this end, fix any $i\in [m]$, and let us consider the terms $\lambda_1 c_{i,1},\ldots, \lambda_n c_{i,n}$ as independent random variables, the $j$th one taking value $c_{i,j}$ with probability $1/2$ and value 0 otherwise, for each $j\in  [n]$. (Note that we are thus using the uniform probability on $\{0,1\}^n$.) Then letting $\mu_i=\frac{1}{2}\sum_{j\in [n]} c_{i,j}$, by Hoeffding's inequality \cite[Chapter 3, Theorem 1.3]{Gut} we have
\[
\forall\, t>0,\qquad \P\left(\Big|\mu_i-\sum_{j\in [n]} \lambda_j\, c_{i,j} \Big| > t\Big(\sum_{j\in [n]} c_{i,j}^2\Big)^{1/2}\right) \leq 2 \exp \left(-2t^2 \right).
\]
Since $\Big(\sum_{j\in [n]} c_{i,j}^2\Big)^{1/2}\leq |D|^{1/2}$, letting $t= \sqrt{ \log( 2 r |S|)/2}$, for $r>0$, we deduce that
\[
\P\left(\Big|\mu_i - \sum_j \lambda_j\, c_{i,j} \Big| >  |D|^{1/2} \sqrt{ \log( 2 r |S|)/2}\right) \leq (r|S|)^{-1} .
\]
By the union bound, the probability that the latter event holds for some $i\in [m]$ is thus at most $r^{-1}$. Hence
\begin{equation}\label{eq:main3}
\P\Big(\Big|\mu_i - (C \lambda)_i \Big|\leq \sqrt{ |D| \log( 2 r |S|)/2}\,\,\textrm{ for all }i\in [m]\Big) \geq 1-r^{-1} .
\end{equation}
 Now let $\Lambda_i=\Big[\mu_i -\sqrt{  |D| \log( 2 r |S|)/2},\mu_i +\sqrt{  |D| \log( 2 r |S|)/2}\Big]$. Combining \eqref{eq:main1} and \eqref{eq:main3}, we obtain that for at least $(1-r^{-1})2^n$ values of $\lambda \in \{0,1\}^n$, the subset sum $\sum_{j\in [n]} \lambda_j d_j$ is an integer linear combination of the elements $s_1,\ldots, s_m$, with $i$th coefficient $(C\lambda)_i\in \Lambda_i$ for each $i\in [m]$. Since these subset sums are pairwise distinct (by dissociativity of $D$), we conclude that
\[
\big(1-r^{-1}\big)\, 2^{|D|} \leq \prod_{j\in [m]} |\Lambda_j |\leq \big(2|D|\log(2r |S|)\big)^{|S|/2}.
\]
Choosing $r=2$, taking $\log_2$ of both sides and rearranging, we obtain \eqref{eq:dslb}.
\end{proof}

We now turn to comparing $d_s$ and $d_d^-$, towards Theorem \ref{thm:midratio}.

We shall call a subset $S$ of an additive set $A$ satisfying $\langle S\rangle \supset A$ and $|S| = d_s(A)$ a \emph{minimum 1-spanning subset of} $A$.

The following small example shows that the dimensions $d_s$ and $d_d^-$ can indeed differ.

\begin{example}\label{lem:eg1}
Let $\{x_1,x_2\}$ be the standard basis in $\R^2$, and let
\[
A=\{x_1, x_2, x_1+x_2, 2x_1, 2x_2\}.
\]
This set has \textup{(}unique\textup{)} minimum 1-spanning subset $\{x_1,x_2,x_1+x_2\}$, while any maximal dissociated subset of $A$ has size $4$. 
\end{example}

The claims in this example are easily checked by inspection. In fact, this example is the simplest case of the following general construction, which is our  main ingredient in our proof of Theorem \ref{thm:midratio}.

\begin{proposition}\label{lem:geneg}
Let $B_n=\{x_1, x_2, \ldots, x_n\}$ be the standard basis of $\R^n$, let $s_n=\sum_{i\in [n]} x_i$, and let $D$ be a dissociated non-empty subset of $\{0,1\}^n$. Then the set
\[
A_n= B_n\cup\{s_n\}\cup (2\cdot D)
\]
satisfies $d_s(A_n)= n+1$ and $d_d^-(A_n)= d_d(A_n) = n+|D|$.
\end{proposition}
Here $2\cdot D$ denotes the set $\{2x:x\in D\}\subset \{0,2\}^n$. 
\begin{proof}
To begin with, we claim that a 1-spanning subset $S\subset A_n$ must have at least $n+1$ elements. To show this, we distinguish two cases.

Case 1: $S$ does not contain $s_n$. Then, in order to be 1-spanning, $S$ must contain all other elements of $A_n$. Indeed, firstly, an element $x_i\in B_n$ must lie in $S$, for otherwise it cannot be in the 1-span of $S$, since every element of $A_n\setminus\{s_n, x_i\}$, modulo 2, has a zero $x_i$-component. An element of $2\cdot D$ must also lie in $S$, for it cannot be in the 1-span of other elements of $2\cdot D$ (since $D$ is dissociated), nor can it lie in $2\cdot D +\varepsilon_1 x_1+\dots+\varepsilon_n x_n$ with $\varepsilon_i\in [-1,1]$ not all zero, as it is congruent to 0 modulo 2. We have thus shown that $S$ must indeed contain $A_n\setminus \{s_n\}$, so our claim holds in this case, i.e. $|S|\geq n+1$.

Case 2: $S$ contains $s_n$, and does not contain some $x_j$. (If it contained $s_n$ and every $x_j$, then our claim would hold already.) In this case, in order to 1-span $x_j$ using $s_n$, the set $S$ must contain every $x_i$ with $i\neq j$. Moreover, $S$ must then also contain every element of $2\cdot D$. Indeed, an element of $2\cdot D$ equals either $2x_j$ or some combination  $y$ involving some $2x_i$ with $i\neq j$. Now $2x_j$ must lie in $S$ in order to be 1-spanned by $S$, since $S$ does not contain $x_j$ and $D$ is dissociated. We claim that $S$ must also contain every other $y\in 2\cdot D$. Indeed, suppose that $y$ were not in $S$, and suppose that we had a $[-1,1]$-combination of elements of $S$ equal to $y$. This combination would then have to involve $s_n$, because otherwise it could only involve elements of $2\cdot D$ different from $y$, contradicting that $2\cdot D$ is dissociated. By involving $s_n$, this combination involves $x_j$. But the latter can then be neither cancelled nor increased to $2x_j$, since $S$ misses $x_j$, whence this combination could not equal $y$, a contradiction. We conclude that $S$ must be $A_n\setminus \{x_j\}$, so we have $|S|=n+|D|\geq n+1$ in this case.

The set $S_n:=B_n \cup \{s_n\}$, of size $n+1$, is 1-spanning for $A_n$ (and is not dissociated). We have thus shown that $d_s(A_n)= n+1$.

Now suppose that $S$ is a maximal dissociated subset of $A_n$. Then $S$ cannot contain $S_n$, so there exists some element $s\in S_n\setminus (S\cap S_n)$. Note also that, being maximal dissociated, $S$ must be 1-spanning for $A_n$. We can then distinguish the same two cases as above.

In the first case, we have $s=s_n$. Then, as in case 1 above, we must have $S= A_n\setminus \{s_n\}$, which is dissociated (as can be seen using that $B_n$ and $2\cdot D$ both are), clearly maximal, and of size $n+|D|$.

In the second case, we have $s = x_j$ for some $j\in [n]$. Then, $S$ must contain $s_n$ (it cannot 1-span it otherwise) and so we are in case 2 above, in which $S$ must be $A_n\setminus \{x_j\}$. Thus in this second case, either we get a contradiction (if $A_n\setminus \{x_j\}$ is not dissociated), or $S=A_n\setminus \{x_j\}$ is a maximal dissociated set of size $n+|D|$.
\end{proof}

We now combine Proposition \ref{lem:geneg} with \cite[Theorem 1]{LY}.
\begin{proof}[Proof of Theorem \ref{thm:midratio}]
As mentioned in the introduction, there exists a dissociated  set $D_n\subset \{0,1\}^n$ of cardinality $|D_n|= n\log_4 n\, (1+o(1))$ as $n\to \infty$. Applying Proposition  \ref{lem:geneg} with this set $D_n$, we obtain a set $A_n\subset \{0,1,2\}^n$ satisfying $d_s(A_n)=n+1$ and $d_d^-(A_n)=d_d(A_n)= n \log_4 n\,(1+o(1)_{n\to \infty})$, whence \eqref{eq:midratio} follows.
\end{proof}

\section{Focusing on some sets of integers: Theorem \ref{thm:[N]}}\label{section:[N]}

So far, the examples that we have discussed of additive sets with small dimension-ratios have all been given by subsets of $\Z^n$ for large $n$. Note that by applying an appropriate Freiman isomorphism of sufficiently high order to such a set, we can obtain a subset of $\Z$ satisfying the same dimensional properties. For example, if for each $n$ we choose a Freiman isomorphism $\phi_n:\{0,1,2\}^n\to \Z$ of order $n^2$ (say) and satisfying\footnote{The existence of such Freiman isomorphisms is a standard result; see for instance \cite[Lemma 5.25]{T-V}.} $\phi_n(0)=0$, then applying $\phi_n$ to the set $A_n$ from Theorem \ref{thm:midratio} for each $n$ we obtain a family of sets $\phi_n(A_n)\subset \Z$ satisfying \eqref{eq:midratio}. One may wonder whether for some natural families of subsets of $\Z$ the dimensions $d_s^-, d_s,d_d^-,d_d$ lie closer to each other. In this section we show that this is the case for the family of intervals $[N]$, in the sense of Theorem \ref{thm:[N]}; thus we have $d_s([N])=d_d^-([N])$ for any positive integer $N$.

To prove Theorem \ref{thm:[N]}, we shall construct a maximal dissociated subset of $[N]$ of size $d_s([N])$, using the following simple fact concerning the powers of 3.
\begin{lemma}\label{lem:P3}
The set $P_3(k)=\{1,3,\ldots, 3^{k-1}\}$ satisfies $\langle P_3(k)\rangle = \left[-\frac{3^k-1}{2},\frac{3^k-1}{2}\right]$.
\end{lemma}

\begin{proof}
The claim holds for $k=1$. For $k>1$, we may suppose by induction that the claim holds for $k-1$, thus $\langle P_3(k-1)\rangle \supset \left[-\frac{3^{k-1}-1}{2},\frac{3^{k-1}-1}{2}\right]$. Then we have
\begin{eqnarray*}
\langle P_3(k)\rangle & = & \{-3^{k-1},0,3^{k-1}\} + \langle P_3(k-1)\rangle  =   \{-3^{k-1},0,3^{k-1}\} + \left[\frac{-3^{k-1}+1}{2},\frac{3^{k-1}-1}{2}\right] \\
& = &  \left[-\frac{3^k-1}{2},\frac{3^k-1}{2}\right] .
\end{eqnarray*}
\end{proof}
We shall also use the following.
\begin{lemma}\label{lem:dissospan}
Let $A$ be an additive set and let $S\subset A$ be  dissociated and satisfy $\langle S \rangle \supset A$. Then $S$ is maximal dissociated.
\end{lemma}
\begin{proof}
If there existed $a\in A\setminus S$ such that $S\cup \{a\}$ is dissociated, then $a$ could not lie in the 1-span of $S$, contradicting that $\langle S \rangle \supset A$.
\end{proof}

To establish Theorem \ref{thm:[N]} we distinguish two cases, according to whether the fractional part $\{ \log_3 N \}:= \log_3 N - \lfloor \log_3 N \rfloor$ satisfies $\{ \log_3 N \}<1-\log_3 2$ or $\{ \log_3 N \}> 1-\log_3 2$.

\begin{proposition}\label{prop:case1}
Let $N$ be a positive integer. The following statements are equivalent.
\begin{enumerate}[leftmargin=30pt]
\item We have $\{ \log_3 N \} <1-\log_3 2$.

\item The set $S_1:=\{1,3,3^2,\ldots, 3^{\lfloor \log_3 N\rfloor}\}$ is a minimum 1-spanning maximal dissociated subset of $[N]$. In particular $d_s([N])= d_d^-([N])= \lfloor\log_3 N\rfloor +1$.
\end{enumerate}
\end{proposition}

\begin{proof}
It follows from Lemma \ref{lem:P3} that
\[
\langle S_1\rangle = \langle P_3(\lfloor \log_3 N\rfloor+1) \rangle = \left[-\frac{3^{\lfloor \log_3 N\rfloor+1}-1}{2},\frac{3^{\lfloor \log_3 N\rfloor+1}-1}{2}\right].
\]
Therefore $S_1$ is a 1-spanning subset of $[N]$ if and only if $\frac{3^{\lfloor \log_3 N\rfloor+1}-1}{2}\geq N$, that is if and only if $\{ \log_3 N \} <1-\log_3 2$. In particular, $(ii)$ implies $(i)$.

Now if $(i)$ holds, then we claim that $S_1$ is in fact  \emph{minimum} 1-spanning for $[N]$. Indeed, any 1-spanning subset $S$ of $[N]$ must satisfy $N \leq (3^{|S|}-1)/2$, since to cover $[N]$ with $[-1,1]$-combinations of $S$ we only use the combinations with positive value. Hence $|S| \geq \log_3 (2N+1) >\log_3 N \geq \lfloor \log_3 N\rfloor$, so we have indeed that $|S| \geq  \lfloor \log_3 N\rfloor +1=|S_1|$. Finally, note that $S_1$ is dissociated, so by Lemma \ref{lem:dissospan} it is maximal dissociated in $[N]$. We have thus shown that $(ii)$ holds.
\end{proof}

We now treat the second case.

\begin{proposition}\label{prop:case2}
Let $N$ be a positive integer, and let $t=1+ \sum_{i=0}^{  \lfloor \log_3 N\rfloor }3^i= \frac{3^{\lfloor \log_3 N \rfloor+1}+1}{2}$. The following statements are equivalent.
\begin{enumerate}[leftmargin=30pt]
\item We have $\{ \log_3 N \} > 1-\log_3 2$.

\item The set $S_2:=\{1,3,3^2,\ldots, 3^{\lfloor \log_3 N\rfloor}\}\cup \{t\}$ is a minimum 1-spanning maximal dissociated subset of $[N]$. In particular $d_s([N])= d_d^-([N])= \lfloor\log_3 N\rfloor +2$.
\end{enumerate}
\end{proposition}
\begin{proof}
By Lemma \ref{lem:P3} we have
\[
\langle S_2\rangle = \langle P_3(\lfloor \log_3 N\rfloor+1) \rangle + \{-t,0,t\} = \left[-3^{\lfloor \log_3 N\rfloor+1},3^{\lfloor \log_3 N\rfloor+1}\right].
\]
Thus $S_2$ is a 1-spanning set for $[N]$ which is dissociated. We have $S_2\subset [N]$ if and only if $t\leq N$, i.e.  $\{ \log_3 N \} > 1-\log_3 2$. In particular, $(ii)$ implies $(i)$.

If $(i)$ holds, then we claim that $S_2$ is \emph{minimum} 1-spanning. Indeed, as shown at the end of the previous proof, if $S$ is 1-spanning for $[N]$ then we must have $|S|\geq \log_3(2N+1)$. If $|S|$ were less than $|S_2|$, i.e. if  $|S|\leq \lfloor\log_3 N\rfloor +1$, then we would have $\lfloor \log_3 N \rfloor +1\geq \log_3(2N+1) > \log_3 2+ \log_3 N$, that is $\{ \log_3 N \}< 1-\log_3 2$, which contradicts $(i)$, so we must have $|S|\geq \lfloor\log_3 N\rfloor +2=|S_2|$. Note also that $S_2$ is dissociated, and therefore maximal dissociated in $[N]$ (by Lemma \ref{lem:dissospan} again). We have thus shown that $(ii)$ holds.
\end{proof}

This completes the proof of Theorem \ref{thm:[N]}.

\section{Final remarks}

In \cite{Sch}, Schoen gave an interesting argument, using Chang's theorem, yielding an upper bound for the maximum density of a subset $A$ of $\Zmod{p}$ ($p$ prime) such that the Cartesian power $A^k$ contains no  element $x$ solving a given integer linear equation $L(x)=c_1x_1+\cdots +c_k x_k=0$. We call such a set $A$ an \emph{$L$-free set}.  Schoen's upper bound involves the dimension $d_s^-(C)$, where $C=\{c_1,\ldots,c_k\}$ is the set of  coefficients of $L$ (in \cite{Sch} this dimension is denoted $\ell(C)$). It is a straightforward task to check that in Schoen's argument one can use $d_s(C)$ instead of $d_s^-(C)$. Thus one obtains the following version of Schoen's result.

\begin{theorem}\label{thm:circle-u-b}
Let $L(x)=c_1x_1+\cdots+c_k x_k$ be a linear form with coefficients $c_i \in \Z$, and let $m_L(\Zmod{p})=\max\{ |A|/p: A\subset \Zmod{p},\;A\textrm{ is }L\textrm{-free}\}$. Then
\begin{equation}\label{eq:main-bounds}
m_L(\Zmod{p})\leq e^{- d_s(C)/12},
\end{equation}
where $C=\{c_1,c_2,\ldots,c_k\}$.
\end{theorem}

As recalled in the introduction, there exists a dissociated set $D\subset \{0,1\}^n$ of size $\sim n\log_4 n$, and this has dimension $d_s(D)=|D|$, which is roughly $\log_4 n$ times $d_s^-(D)=n$. Applying an appropriate Freiman isomorphism $\phi: \{0,1\}^n \to \Z$, as in the previous section, we obtain a set $C=\phi(D) \subset \Z$ with the same properties (note that $d_s^-(C)\leq d_s^-(D)$ and $d_s(C)=d_s(D)$). For a linear form $L$ with coefficient-set $C$,  the bound \eqref{eq:main-bounds} is thus stronger than the version with $d_s^-(C)$. It would be interesting to strengthen the upper bound on $m_L(\Zmod{p})$ further.

\textbf{Acknowledgements.} The first author is grateful to Jakob Vidmar for programming computer searches that shed light on problems treated in this paper. The authors are also very grateful to Vsevolod Lev for bringing to their attention the results on dissociated subsets of $\{0,1\}^n$ in \cite{Bs,C&M,Lind}.

\end{document}